\documentclass[12pt, eqno, twoside]{article}
\usepackage{graphicx}
\usepackage{amscd}
\usepackage[matrix,arrow,curve]{xy}
\usepackage{amssymb,amsmath}
\usepackage{amsmath,amsfonts,amssymb,latexsym}
\usepackage{fancyhdr}
\usepackage{colordvi}
\usepackage{color}

\newtheorem{theorem}{Theorem}[section]

\newtheorem{definition}[theorem]{Definition}

\newtheorem{lemma}[theorem]{Lemma}

\newenvironment{proof}[1][Proof]{\noindent\textbf{#1.} }{\
\rule{0.5em}{0.5em}}

\begin{document}

\pagestyle{fancy}
\fancyhead{} 
\fancyhead[EC]{\small\it P.P.
Ntumba}%
\fancyhead[EL,OR]{\thepage} \fancyhead[OC]{\small\it
On the endomorphisms of some sheaves of functions}%
\fancyfoot{} 
\renewcommand\headrulewidth{0.5pt}
\addtolength{\headheight}{2pt} 

\title{\textbf{On the endomorphisms of some sheaves of functions}}
\author{P.P. Ntumba\footnote{I am indebted to Professor P. Schapira for his many constructive remarks, which led to the present form of the article.}}

\date{}
\maketitle

\begin{abstract}
Given a $C^\infty$ real manifold $X$ and $\mathcal{C}^m_X$ its sheaf of  $m$-times differentiable real-valued functions, we prove that the sheaf $\mathcal{D}^{m, r}_X$ of differential operators of order $\leq m$ with coefficient functions of class $C^r$ can be obtained in terms of the sheaf $\mathcal{H}om_{\mathbb{R}_X}(\mathcal{C}^m_X, \mathcal{C}^r_X)$ of morphisms of $\mathcal{C}^m_X$ into $\mathcal{C}^r_X$. The superscripts $m$ and $r$ are integers.
\end{abstract}

{\it Key Words}: Distributions, support of a function, sheaf of functions, sheaf of differential operators, partition of unity, functor $\mathcal{H}om$.

\section{Statement of the main theorem}

In this note, we denote by $X$ a real $C^\infty$ manifold and by $\mathcal{C}^m_X$ the \textit{sheaf of real-valued functions of class $C^m$}, on $X$, with $m\in \mathbb{N}\cup \{0\}$. On the other hand, we denote by $\mathcal{D}^{m, r}_X$, where $0\leq m< \infty$ and $0\leq r< \infty$, the \textit{sheaf of differential operators of order $\leq m$ with coefficients of class $C^r$}. Note that $\mathcal{D}^{0, r}_X= \mathcal{C}^r_X$, for any nonnegative integer $r$. Recall that a section ${P}$ on $\mathcal{D}^{m, r}_X$ may be written locally in a local chart as (see \cite[p.13]{hormander})
\begin{equation}
\label{eq6}
{P}= \sum_{\vert\alpha\vert\leq m}a_\alpha(x)\partial_X^\alpha,
\end{equation}
where $a_\alpha$ are real-valued functions of class $C^r$. In the equation (\ref{eq6}), we have used the following classical notation: $\partial_i:= \frac{\partial}{\partial x_i}$, where $i=1, \ldots, n=\dim X$, $\alpha:= (\alpha_1, \ldots, \alpha_n)$ is a multi-index, and $\partial^\alpha:= \partial^{\alpha_1}_1\cdots \partial^{\alpha_n}_n$. Furthermore, we also set, by classical convention, that $\vert\alpha\vert:= \sum \alpha_i$ and $\alpha!:= \alpha_1!\cdots \alpha_n!$.

\begin{definition}
\label{definition1}
\emph{For any $x_0\in X$, we denote by $\mathfrak{M}^m_{x_0}\subseteq \mathcal{C}^m_X$ the subsheaf of $\mathcal{C}^m_X$, given by
\[
\mathfrak{M}_{x_0}^m(U)= \left\{\begin{array}{ll} \mathcal{C}^m_X(U), \ \ \mbox{if $x_0\notin U$} & {}\\
\varphi\in \mathcal{C}^m_X(U):\ \mbox{for all $\vert\alpha\vert\leq m$, $(\partial^\alpha_U\varphi)(x_0)=0$}, & \mbox{if $x_0\in U$.}
\end{array}\right.
\]}
\end{definition}

Notice that in the equation $(\partial^\alpha_U\varphi)(x_0)=0$ of Definition \ref{definition1}, we have assumed that $U$ is a local chart of $x_0$.


For any nonnegative integers $m$ and $r$ such that $m\geq r$, the morphism
\begin{equation}
\label{eq7}
\begin{array}{ll}
\mathcal{D}^{m-r, r}_X\longrightarrow \mathcal{H}om_{\mathbb{R}_X}(\mathcal{C}^m_X, \mathcal{C}^r_X), & P\longmapsto (\varphi\longmapsto P(\varphi))
\end{array}
\end{equation}
is called the \textit{natural morphism of $\mathcal{D}^{m-r, r}_X$ into $\mathcal{H}om_{\mathbb{R}_X}(\mathcal{C}^m_X, \mathcal{C}^r_X)$.} On the other hand, we set that \textit{if $m-r<0$, then $\mathcal{D}^{m-r, r}=0.$}

This is our main theorem.

\begin{theorem}
\label{theorem3}
For any $m-r\in \mathbb{Z}\cup \{-\infty\}$, the morphism $(\ref{eq7})$ is an isomorphism.
\end{theorem}

Theorem \ref{theorem3} associates naturally with Peetre's theorem (\cite{peetre1, peetre2}), which states the following: \textit{Let $\Omega$ be an open set in $\mathbb{R}^n$, $\mathcal D$ the sheaf of $C^\infty$ functions on $\mathbb{R}^n$, and $\mathcal{D}^\prime$ the sheaf of distributions on $\mathbb{R}^n$. Moreover, let $P: \mathcal{D}(\Omega)\longrightarrow \mathcal{D}^\prime(\Omega)$ be a linear map such that 
\begin{equation}
\label{eq18}
supp(Pf)\subseteq supp(f),
\end{equation}
and let $\Lambda$ be the set of points of discontinuity of $P$. Then, there exists a family of distributions $\{a^\alpha\}$, locally finite and unique outside of $\Lambda$, such that 
\[
supp(Pf-\sum a^\alpha D_\alpha f)\subseteq \Lambda,
\]
for every $f\in \mathcal{D}(\Omega)$.}

This theorem gives rise to another theorem, also proved by Peetre in \cite{peetre2}: \textit{Let $H$ be a subspace of $\mathcal{D}^\prime(\Omega)$, closed under multiplication with elements of $\mathcal{D}(\Omega)$, and assume that $H$ has no elements with finite local support. Moreover, let $P: \mathcal{D}(\Omega)\longrightarrow H$ be a linear map satisfying (\ref{eq18}). Then, there exists a unique family of distributions $\{a^\alpha\}$, locally finite and locally in $H$, such that 
\[
Pf= \sum a^\alpha D_\alpha f,
\]
for every $f\in \mathcal{D}(\Omega)$.}

\section{Proof of Theorem \ref{theorem3}}

First, let us recall the following classical result (see, for instance \cite[p.5, Lemma 1.1.1]{jost}).

\begin{lemma}
\label{theorem1}
Let $\{U_i\}_{i\in I}$ be a finite open covering of the unit sphere $S^{n-1}\subseteq \mathbb{R}^n$. Then, there exists a family of nonnegative real-valued maps of class $C^\infty$ $\varphi_i: S^{n-1}\longrightarrow \mathbb{R}$, such that
\begin{enumerate}
\item [{$(i)$}] $\mbox{supp $\varphi_i\subseteq U_i$, for all $i$,}$
\item [{$(ii)$}] $0\leq \varphi_i(x)\leq 1$, $\mbox{for all $x\in S^{n-1}$, $i\in I$}$
\item [{$(iii)$}] $\sum_{i\in I}\varphi_i(x)=1,$ $\mbox{for all $x\in S^{n-1}$}$.
\end{enumerate}
\end{lemma}

Still assuming the notations of Lemma \ref{theorem1}, for every $i\in I$, we let $\psi_i: \mathbb{R}^n\setminus \{0\}\longrightarrow \mathbb R$ be the map given by
\begin{equation}
\label{eq4}
\psi_i(x)= \varphi_i\big(\frac{x}{||x||}\big).
\end{equation}
Clearly, $\psi_i$ is $C^\infty$ on $\mathbb{R}^n\setminus \{0\}$ and $\vert\vert\partial^\alpha\psi_i\vert\vert$ is bounded for any $\vert\alpha\vert<\infty.$
Next, let $m\in \mathbb{N}$ and $\varphi: \mathbb{R}^n\longrightarrow \mathbb R$ be a $C^m$ real-valued function such that $\varphi(0)=0$. For every $i\in I$, set
\begin{equation}
\label{eq2}
(\psi_i\varphi)(x)= \left\{\begin{array}{ll}\psi_i(x)\varphi(x), & \mbox{if $x\neq 0$}\\ 0, & \mbox{if $x=0$.}\end{array}\right.
\end{equation}
It is clear that
\begin{equation}
\label{eq1}
\varphi= \sum_{i\in I}\psi_i\varphi.
\end{equation}

Thus, we have

\begin{lemma}
\label{lemma1}
Let $U$ be an open neighborhood of $0$ in $\mathbb{R}^n$. For any nonnegative integer $m$, if $\varphi\in \mathfrak{M}^m_0(U)$, then ${\psi_i}\varphi\in \mathfrak{M}^0_0(U),$ where the $\psi_i$ are the functions \emph{(\ref{eq4})}.
\end{lemma}

Furthermore,

\begin{lemma}
\label{lemma2}
For any open neighborhood $U$ of $0$ in $\mathbb{R}^n$ and nonnegative integer $m$, if $u\in \mathcal{H}om_{\mathbb{R}_{\mathbb{R}^n}}(\mathcal{C}^m_{\mathbb{R}^n}, \mathcal{C}^0_{\mathbb{R}^n})(U)$, then
$u({\mathfrak{M}^m_0})\subseteq {\mathfrak{M}^0_0}.$
\end{lemma}

\begin{proof}
First, let us consider a sphere $S$ contained in $U$ and centered at the origin, and denote, for example, by $A$ and $B$ the north and south poles of $S$. Next, consider the following open covering of $S$: $\{U_1, U_2\}$, where $U_1$ contains $A$ and does not intersect some open neighborhood $V_1$ of $B$, and, similarly, $U_2$ contains $B$ and does not intersect some open neighborhood $V_2$ of $A$. By Lemma \ref{theorem1}, we let $\{\varphi_1, \varphi_2\}$ be a partition of unity subordinate to the covering $\{U_1, U_2\}$, and let $\psi_1, \psi_2$ be functions derived from the $\varphi_i$ as in (\ref{eq4}). We denote by $\mathbb{R}^+V_i$ the open cone generated by $V_i$, $i=1, 2$. It is obvious that $\psi_i$ vanishes on $\mathbb{R}^+V_i$. Now, let $\varphi\in \mathfrak{M}^m_0(U)$; it is easily seen that $\partial^\alpha(\psi_i\varphi)|_{\mathbb{R}^+V_i}=0$, $i=1, 2$, for any $\vert\alpha\vert\leq m$. Since $\vert\vert\partial^\alpha\psi_i\vert\vert$, $i=1, 2$, is bounded, it follows that
\begin{equation}
\label{eq11}
\partial^\alpha({\psi_i}\varphi)|_{\overline{\mathbb{R}^+V_i}}=0.
\end{equation}
Furthermore, since $u$ is linear, Equation (\ref{eq11}) implies that \[
u({\psi_i}\varphi)|_{\overline{\mathbb{R}^+V_i}}=0;
\]
thus
\[
u(\psi_i\varphi)(0)=0,
\]
for every $i=1, 2$. On the other hand, by Lemma \ref{lemma1}, $\psi_i\varphi\in \mathfrak{M}^0_0(U)$, for every $i$; but
\[
\varphi= \sum_{i=1}^2\psi_i\varphi,
\]
therefore
\[
u(\varphi)(0)= u(\psi_1\varphi)(0)+ u(\psi_2\varphi)(0)=0.
\]
Thus,
\[
u(\varphi)\in \mathfrak{M}^0_0(U),
\]
and the proof is finished.
\end{proof}


%

We are now set for the proof of a \textit{particular case of Theorem \ref{theorem3}}: the \textit{isomorphism} $\mathcal{D}_X^{m, 0}\simeq \mathcal{H}om_{\mathbb{R}_X}(\mathcal{C}^m_X, \mathcal{C}^0_X)$. But first let us recall Taylor's theorem for multivariate functions: Let $f: \mathbb{R}^n\longrightarrow \mathbb R$ be $C^k$ at a point $a\in \mathbb{R}^n$. Then there exists $h_\alpha: \mathbb{R}^n\longrightarrow \mathbb R$ such that
\begin{equation}
\label{eq13}
f(x)= \sum_{\vert\alpha\vert\leq k}\frac{\partial^\alpha f(a)}{\alpha!}(x-a)^\alpha + \sum_{\vert\alpha\vert\leq k}h_\alpha(x)(x- a)^\alpha,
\end{equation}

and $\lim_{x\rightarrow a}h_\alpha(x)=0.$ Moreover, if $f$ is $C^{k+1}$, then

\begin{equation}
\label{eq14}
f(x)= \sum_{\vert\alpha\vert\leq k}\frac{\partial^\alpha f(a)}{\alpha!}(x- a)^\alpha+ \sum_{\vert \beta\vert= k+1}R_\beta(x)(x-a)^\beta,
\end{equation}
where
\begin{equation}
\label{eq15}
R_\beta(x)= \frac{\vert\beta\vert}{\beta!}\int^1_0(1-t)^{\vert\beta\vert-1}\partial^\beta f(a+ t(x-a))dt.
\end{equation}
%
%

We have used in both equations (\ref{eq14}) and (\ref{eq15}) the multi-index notation:
\[
x^\alpha= x_1^{\alpha_1}\cdots x_n^{\alpha_n},
\]
where $x= (x_1, \ldots, x_n)$ and $\alpha= (\alpha_1, \ldots, \alpha_n)$.

\begin{lemma}
\label{lemma5}
Let $X$ be an $n$-dimensional real $C^\infty$ manifold. For any local chart $U\subseteq X$, let $u\in \mathcal{H}om_{\mathbb{R}_X}(\mathcal{C}^m_X, \mathcal{C}^0_X)(U)$, $\mathcal{P}^m$ the space of polynomials on $U$ of degree $\leq m$  and $\mathcal{P}^m_U$ the constant sheaf on $U$ with stalk $\mathcal{P}^m$. If $u(\mathcal{P}^m_U)=0,$ then $u=0$.
\end{lemma}

\begin{proof}
Let $\varphi\in \mathcal{C}^m_X(V)$, where $V$ is a subopen of $U$ containing a point $x_0$. By applying mutatis mutandis Taylor's formula (cf. (\ref{eq14})) on $\varphi$, about $x_0$, we have that
\[
\varphi(x_0+h)= q(x_0+h)+ \psi(x_0+ h) ,
\]
where
\[
q(x_0+h)= \sum_{\vert\alpha\vert\leq m-1}\frac{\partial^\alpha\varphi(x_0)}{\alpha!}h^\alpha,
\]
and
\[
\psi(x_0+h)= \sum_{\vert\beta\vert=m}R_\beta(x_0+h)h^\beta;
\]
so $q\in \mathcal{P}_U^m(V)$ and $\psi\in \mathfrak{M}^m_{x_0}(V)$. Then, by virtue of the hypothesis and Lemma \ref{lemma2}, $u(\varphi)\in \mathfrak{M}^0_{x_0}(V)$; therefore $u(\varphi)(x_0)=0$. But this holds for all $x_0\in V$, subopen $V$ of $U$, and $\varphi\in \mathcal{C}^m_X(V)$, so $u=0$.
\end{proof}

\subsection{Case $0\leq r< m$}

\begin{theorem}
\label{theorem5}
Let $X$ be an $n$-dimensional $C^\infty$ real manifold and $\mathcal{D}^{m, r}_X$ the sheaf of differential operators of order $\leq m$ and whose coefficients are of class $C^r$. Then, the natural morphism
\begin{equation}
\label{eq5}
\mathcal{D}^{m, r}_X\longrightarrow \mathcal{H}om_{\mathbb{R}_X}(\mathcal{C}^m_X, \mathcal{C}^r_X)
\end{equation} is an isomorphism.
\end{theorem}

\begin{proof}
That the morphism (\ref{eq5}) is injective is obvious. Let us show surjectiveness. To this end, let $u\in \mathcal{H}om_{\mathbb{R}_X}(\mathcal{C}^m_X, \mathcal{C}^r_X)(U)$, where $U$ is a local chart $(U; x_1, \ldots, x_n)$ in $X$. We will show that $u$ is in fact a differential operator with coefficient functions of class $C^r$. For this purpose, consider the differential operator
\begin{equation}
\label{eq8}
P= \sum_{\vert\beta\vert\leq m}a_\beta(x)\partial^\beta
\end{equation}
in which the $a_\beta$ are of class $C^r$ and defined by induction on $\vert\beta\vert$, in the following manner: let $\mathbb{I}:U\longrightarrow \mathbb{R}$ be the constant function $\mathbb{I}(x)=1$, we set
\[
a_0(x)= u(\mathbb{I})\equiv a_0.
\]
Suppose that we have defined $a_\beta$ for $\vert\beta\vert< t$, then, for any multi-index $\alpha$ such that $\vert\alpha\vert= t$, set:
\[
a_\alpha(x)= (u- \sum_{\vert\beta\vert<t}a_\beta(x)\partial^\beta)(x^\alpha).
\]
Clearly, $a_\alpha\in \mathcal{C}^r_X(U)$. That $P\in \mathcal{H}om_{\mathbb{R}_X}(\mathcal{C}^m_X, \mathcal{C}^r_X)(U)$ is clear. Since, for every $x^\beta$, with $\vert\beta\vert\leq m$,
\[
(u-{P})(x^{\beta})=0,
\]
it follows that
\[
(u-{P})(\mathcal{P}^m_U)=0;
\]
hence, by Lemma \ref{lemma5}, $u=P$, and the proof is finished.
\end{proof}

\subsection{Case $m-r<0$}

\begin{lemma}
For any nonnegative integer $m$, $\mathcal{H}om_{\mathbb{R}_X}(\mathcal{C}^m_X, \mathcal{C}^{m+1}_X)=0.$
\end{lemma}

\begin{proof}
Since $\mathcal{C}^{m+1}_X\subseteq \mathcal{C}^0_X$, we have that
\[
\mathcal{H}om_{\mathbb{R}_X}(\mathcal{C}^m_X, \mathcal{C}^{m+1}_X)\subseteq \mathcal{H}om_{\mathbb{R}_X}(\mathcal{C}^m_X, \mathcal{C}^0_X).
\]
By Theorem \ref{theorem5}, for any open set $U$ in $X$, if
\[
u\in \mathcal{H}om_{\mathbb{R}_X}(\mathcal{C}^m_X, \mathcal{C}^{m+1}_X)(U)\subseteq \mathcal{H}om_{\mathbb{R}_X}(\mathcal{C}^m_X, \mathcal{C}^0_X)(U),
\]
then
\[
u\in \mathcal{D}^{m, 0}_X(U).
\]
Without loss of generality, suppose that $U$ is a local chart of $X$, so
\begin{equation}
\label{eq10}
u= \sum_{\vert\alpha\vert\leq m}a_\alpha(x)\partial^\alpha,
\end{equation}
where $a_\alpha(x)$ is $C^0$. If all the $a_\alpha(x)$, in (\ref{eq10}), are zero , then there is nothing to prove. Suppose there exists an $\alpha$ such that $a_\alpha(x_0)\neq 0$ for some point $x_0\in U$, and let $b\in \mathcal{C}^m_X(U)$ be such that $\partial^{m+1}b$ does not exist at $x_0$; consequently, $b\notin \mathcal{C}^{m+1}_X(U)$. By applying (\ref{eq10}) on $b$, we notice that $u(b)\notin \mathcal{C}^{m+1}_X(U)$. For, suppose, without loss of generality, that, in (\ref{eq10}), $a_\alpha(x)=0$ for all $\alpha$ such that $\vert\alpha\vert<m$; so
\[
u(b)= \sum_{\vert\alpha\vert=m}a_\alpha(x)\partial^\alpha b.
\]
Since $\partial^{m+1}b$ does not exist at $x_0$, then, for some $\beta$ with $\vert\beta\vert=1$,
\[
\partial^\beta(u(b))= \sum_{\vert\alpha\vert=m}\partial^\beta(a_\alpha(x)\partial^\alpha b);
\]
thus $\partial^\beta(a_\alpha(x)\partial^\alpha b)$ does not exist at $x_0$. But this leads to a contradiction as by hypothesis $u(b)\in \mathcal{C}^{m+1}_X(U).$
\end{proof}

Thus, whenever $m< r$, the following theorem holds.

\begin{theorem}
Let $m$, $r\in \mathbb Z$ such that $0\leq m<\infty$, $0\leq r\leq \infty$ and $m-r<0$. Then,
\begin{equation}
\label{eq16}
\mathcal{H}om_{\mathbb{R}_X}(\mathcal{C}^m_X, \mathcal{C}^r_X)=0.
\end{equation}
Hence,
\begin{equation}
\label{eq17}
\mathcal{D}^{m-r, r}_X\simeq \mathcal{H}om_{\mathbb{R}_X}(\mathcal{C}^m_X, \mathcal{C}^r_X).
\end{equation}
\end{theorem}

\begin{proof}
(\ref{eq16}) follows from the sequence
\[
\mathcal{H}om_{\mathbb{R}_X}(\mathcal{C}^m_X, \mathcal{C}^r_X)\subseteq \mathcal{H}om_{\mathbb{R}_X}(\mathcal{C}^m_X, \mathcal{C}^{m+1}_X)=0.
\]
Hence, (\ref{eq17}) holds.
\end{proof}

\subsection{Case $m=r$}

\begin{lemma}
\label{lemma8}
Let $U$ be a local chart in a manifold $X$, $m$ a nonnegative integer and $P: \mathcal{C}^m_X(U)\longrightarrow \mathcal{C}^m_X(U)$ a morphism, given by
\[
P= \sum_{\vert\alpha\vert\leq m}a_\alpha(x)\partial^\alpha_X,
\]
where $a_\alpha(x)\in \mathcal{C}^0_X$. Then, $a_0(x)\in \mathcal{C}_X^m(U)$ and $a_\alpha(x)=0$ for all $\alpha$ such that $\vert\alpha\vert\neq 0$. In other words,
\[
\mathcal{H}om_{\mathbb{R}_X}(\mathcal{C}^m_X, \mathcal{C}^m_X)\subseteq \mathcal{C}^m_X.
\]
\end{lemma}

\begin{proof}
Let $\mathbb{I}:U\longrightarrow \mathbb{R}$ denote the constant function $\mathbb{I}(x)=1$. Then,
\[
P(\mathbb{I})= a_0(x)\equiv a_0
\]
and
\[
a_0\in \mathcal{C}^m_X(U).
\]
For any $\alpha= (\underbrace{0, \cdots, 0}_{i-1}, \alpha_i, 0, \cdots, 0)$, $1\leq i\leq n$, we contend that $a_\alpha=0$. In fact, let $\varphi\equiv \varphi(x_i)= x_i^m\in \mathcal{C}^{m-1}_X(U)$; if $\widetilde{\varphi}$ is an antiderivative of $\varphi$, then $\widetilde{\varphi}\in \mathcal{C}^m_X(U)$ and
\begin{equation}
\label{eq3}
\sum_{\alpha_i=1}^ma_{(0, \cdots, 0, \alpha_i, 0, \cdots, 0)}(x)\frac{\partial^{\alpha_i}\widetilde{\varphi}(x_i)}{\partial x_i^{\alpha_i}}= (P-a_0)(\widetilde{\varphi}(x_i))
\end{equation}
with
\[
(P-a_0)(\widetilde{\varphi}(x_i))\in \mathcal{C}^m_X(U).
\]
But the collection $\{\widetilde{\varphi}(x_i), \frac{\partial \widetilde{\varphi}(x_i)}{\partial x_i}, \ldots, \frac{\partial^m \widetilde{\varphi}(x_i)}{\partial x_i^m}\}$ is linearly independent, therefore, we have that
\[
P= a_0
\]
and
\[
a_{(0, \cdots, 0, \alpha_i, 0, \cdots, 0)}(x)=0
\]
for all $1\leq \alpha_i\leq m$.

Now, let $\alpha= (\underbrace{0, \cdots, 0}_{i-1}, \alpha_i, \alpha_{i+1}, 0, \cdots, 0)$ with $1\leq \alpha_i+ \alpha_{i+1}\leq m$. Consider the function $\psi\equiv \psi(x_i, x_{i+1})= x_i^mx_{i+1}^m\in \mathcal{C}^m_X(U)$ and let $\widetilde{\psi}\equiv \widetilde{\psi}(x_i, x_{i+1})$ be such that
\[
\frac{\partial^2\widetilde{\psi}(x_i, x_{i+1})}{\partial x_i\partial x_{i+1}}= \psi(x_i, x_{i+1}).
\]
Then,
\begin{equation}
\sum_{\alpha_{i+1}}^m\sum_{\alpha_i}^ma_{(0, \cdots, 0, \alpha_i, \alpha_{i+1}, 0, \cdots, 0)}(x)\frac{\partial^{\alpha_i+\alpha_{i+1}}\widetilde{\psi}(x_i, x_{i+1})}{\partial x_i^{\alpha_i}\partial x_{i+1}^{\alpha_{i+1}}}= (P- a_0)(\widetilde{\psi}(x_i, x_{i+1}))
\end{equation}
with
\[
(P- a_0)(\widetilde{\psi}(x_i, x_{i+1}))\in \mathcal{C}^m_X(U).
\]
Since the collection
\[
\{\widetilde{\psi}(x_i, x_{i+1}), \frac{\partial^{\alpha_i+\alpha_{i+1}}\widetilde{\psi}(x_i, x_{i+1})}{\partial x_i^{\alpha_i}\partial x_{i+1}^{\alpha_{i+1}}}\}_{1\leq \alpha_i+ \alpha_{i+1}\leq m}
\]
is linearly independent, it follows that
\[
a_{(0, \cdots, 0, \alpha_i, \alpha_{i+1}, 0, \cdots, 0)}(x)=0
\]
for all $\alpha_i$ and $\alpha_{i+1}$ such that $1\leq \alpha_i+ \alpha_{i+1}\leq m$.

In an analogous way, one shows that the rest of the coefficient functions $a_\alpha(x)$ outside of $a_0(x)$ are $0$. The proof is thus finished.
\end{proof}

That $\mathcal{C}^m_X\subseteq \mathcal{H}om_{\mathbb{R}_X}(\mathcal{C}^m_X, \mathcal{C}^m_X)$ is obvious. In fact, for any open $U\subseteq X$ and $\varphi\in \mathcal{C}^m_X(U)$, the map $\varphi: {\mathcal{C}^m_X}|_U\longrightarrow {\mathcal{C}^m_X}|_U$, defined by
\[
\varphi(\psi)= \varphi \psi
\]
is linear. Thus, on considering Lemma \ref{lemma8}, we have

\begin{theorem}
Let $m\in \mathbb{N}\cup \{0\}$. Then,
\[
\mathcal{H}om_{\mathbb{R}_X}(\mathcal{C}^m_X, \mathcal{C}^m_X)\simeq \mathcal{C}^m_X.
\]
\end{theorem}

Let us now state diagrammatically our results.

\begin{theorem}
Let $X$ be an $n$-dimensional $C^\infty$ real manifold and $\mathcal{D}^{m, r}_X$ the sheaf of differential operators of order $\leq m$ and whose coefficients are of class $C^r$. Then, the diagram
\[
\xymatrix{\mathcal{D}^{m-n, n}_X\ar[r]\ar@{^{(}->}[d] & \mathcal{H}om_{\mathbb{R}_X}(\mathcal{C}^m_X, \mathcal{C}^n_X)\ar@{^{(}->}[d]\\ \mathcal{D}^{m, k}_X\ar[r] & \mathcal{H}om_{\mathbb{R}_X}(\mathcal{C}^m_X, \mathcal{C}^k_X),}
\]
where horizontal arrows are isomorphisms, vertical arrows injections, and $k< m\leq n$, commutes.
\end{theorem}

%
%

\noindent P.P. Ntumba\\{Department of Mathematics and
Applied Mathematics}\\{University of Pretoria}\\ {Hatfield 0002,
Republic of South Africa}\\{Email: patrice.ntumba@up.ac.za}


\end{document}